\documentclass{article}%

\usepackage{graphicx}
\usepackage{amsmath,amssymb,amsfonts}%
\usepackage{theorem}%
\usepackage{color}%
\usepackage{hyperref}
\usepackage[font={small, it}]{caption}
\theoremstyle{change}%
\newtheorem{definition}{Definition:}[section]%
\newtheorem{proposition}[definition]{Proposition:}%
\newtheorem{theorem}[definition]{Theorem:}%
\newtheorem{lemma}[definition]{Lemma:}%
\newtheorem{corollary}[definition]{Corollary:}%
{\theorembodyfont{\rmfamily}\newtheorem{remark}[definition]{Remark:}}%
{\theorembodyfont{\rmfamily}\newtheorem{example}[definition]{Example:}}%

\newenvironment{proof}
{{\bf Proof:}}
{\qquad \hspace*{\fill} $\Box$}%

\newcommand{\fg}{\mathfrak{g}}%
\newcommand{\fs}{\mathfrak{s}}%
\newcommand{\fz}{\mathfrak{z}}%

\newcommand{\Ad}{\operatorname{Ad}}%
\newcommand{\ad}{\operatorname{ad}}%
\newcommand{\tr}{\operatorname{tr}}%
\newcommand{\id}{\operatorname{id}}
\newcommand{\cl}{\operatorname{cl}}%
\newcommand{\fix}{\operatorname{fix}}%
\newcommand{\rme}{\mathrm{e}}%

\newcommand{\dds}{\frac{d}{ds}}

\newcommand{\EC}{\mathcal{E}}%
\newcommand{\OC}{\mathcal{O}}%
\newcommand{\RC}{\mathcal{R}}%
\newcommand{\YC}{\mathcal{Y}}%
\newcommand{\XC}{\mathcal{X}}%
\newcommand{\DC}{\mathcal{D}}%
\newcommand{\NC}{\mathcal{N}}%
\newcommand{\HC}{\mathcal{H}}%

\newcommand{\T}{\mathbb{T}}%
\newcommand{\C}{\mathbb{C}}%
\newcommand{\N}{\mathbb{N}}%
\newcommand{\R}{\mathbb{R}}%
\newcommand{\Z}{\mathbb{Z}}%

\begin{document}

\title{Jordan decomposition and the recurrent set of flows of automorphisms}
\author{V\'{\i}ctor Ayala%
\thanks{
Supported by Proyecto Fondecyt n$%
{{}^\circ}%
$ 1190142. Conicyt, Chile.} \\
Instituto de Alta Investigaci\'{o}n\\
Universidad de Tarapac\'{a}, Arica, Chile \and Adriano Da Silva\thanks{ Supported by Fapesp grant $n^{o}$ 2018/10696-6} \\
Instituto de Matem\'{a}tica\\
Universidade Estadual de Campinas, Brazil
\and Philippe Jouan \\ 
Laboratoire de Math\'{e}matiques Rapha\"{e}l Salem\\
Universit\'{e} de Rouen, France\\
}
\date{\today }
\maketitle

\begin{abstract}
In this paper we show that any linear vector field $\XC$ on a connected Lie group $G$ admits a Jordan
decomposition and the recurrent set of the associated flow of automorphisms is given as the
intersection of the fixed points of the hyperbolic and nilpotent components of its Jordan
decomposition.
\end{abstract}



\section{Introduction}

Many relevant applications are coming from physical problems where the state space is a Lie group. For instance, the Noether Theorem \cite{Noether}, states
that every differentiable symmetry of the action of a physical system has a corresponding conservation law. And, it is possible to associate symmetry with
dynamic through the notion of invariant vector fields on Lie groups. Furthermore, since the remarkable Brockett's paper "Systems theory on group manifolds and cosset spaces", in 1972, many people have been working in control theory from the geometric point of view. Specially, on invariant control systems. On the other hand, in \cite{VAJT}
the authors introduced the class of linear control systems on Lie groups, which is determined by a linear vector field $\mathcal{X}$ as a drift, controled by a number of invariant vector fields. This class of system is a perfect generalization of the classical linear control systems on Euclidean
space, and it is relevant from both, theoretical and practical point of view (see \cite{JPh1}). Recently, this class have been associated with
the almost-Riemannian structures (ARS) (see \cite{JPh2}). The analysis of any linear system and each ARS on a Lie group, depends strongly of the
dynamics of the drift (see for instance \cite{DSAy, DSAy1, DSAyGZ, VAASPJ, DS}). The main aim of this paper is to give a
contribution to understand the dynamic of $\mathcal{X}$. In order to do that, we first study the recurrent points of $\mathcal{X}$ and prove that any
linear vector field admits a mutiplicative Jordan decomposition of its flow, more precisely.


A vector field $\mathcal{X}$ on $G$ is said to be
\textit{linear} if its flow $\{ \varphi_{t}\}_{t\in \mathbb{R}}$ is a
$1$-parameter subgroup of $\mathrm{Aut}(G)$. Associated to $\mathcal{X}$ there
is a derivation $\DC:\mathfrak{g}\rightarrow \mathfrak{g}$ of $\mathfrak{g}$ that
satisfies
\[
(d\varphi_{t})_{e}=\mathrm{e}^{t\mathcal{D}}\; \; \; \mbox{ for all }\; \; \;t\in
\mathbb{R}.
\]
Through the Lyapunov spectrum of $\DC$ we consider its well known additive Jordan decomposition into elliptic, hyperbolic and nilpotent parts. Then, we say $\mathcal{X}$ admits a Jordan decomposition if it can be written as a the sum of commutitive vector fields
\[
\mathcal{X}=\mathcal{X}_{\mathcal{E}}+\mathcal{X}_{\mathcal{H}}+\mathcal{X}%
_{\mathcal{N}},
\]
where $\mathcal{X}_{\mathcal{E}}$, $\mathcal{X}_{\mathcal{H}}$ and $\mathcal{X}_{\mathcal{N}}$ are linear vector fields whose associated derivation are the elliptic, hyperbolic and nilpotent parts of $\DC$, respectively. Although any linear map admits a Jordan decomposition, the same decomposition for linear vector fields depends on the integration of automorphisms of $\fg$ to automorphism of $G$, which is in general not always possible. However, if a linear vector field admits
a Jordan decomposition and we denote by $\{\varphi_{t}^{\mathcal{H}}\}_{t\in\R}$ and $\{\varphi_{t}^{\mathcal{N}}\}_{t\in\R}$ the flows of $\XC_{\HC}$ and $\XC_{\NC}$, respectively, then 

\begin{theorem}
	It holds that
	\[
	\mathcal{R}(\varphi_{t})=\operatorname{fix}(\varphi_{t}^{\mathcal{H}}%
	)\cap \operatorname{fix}(\varphi_{t}^{\mathcal{N}}).
	\]
	
\end{theorem}

We build the proof by considering the cases where the toral component $T(G)$
of $G$ is trivial or not. Here $T(G)$ is the maximal compact, connected
subgroup of $Z(G)_{0}$. The knowledge of $\RC(\varphi_t)$ and the fact that any linear vector field on simply connected Lie groups admits a Jordan decomposition allows us to prove that in fact,  

\begin{theorem}
	Any linear vector field $\mathcal{X}$ on a connected Lie group $G$ admits a
	Jordan decomposition.
\end{theorem}

In particular, the set of recurrent points of any linear vector field $\mathcal{X}$ is completely characterized by the hyperbolic and nilpotent parts of its Jordan decomposition. Finally, through the so called Arnold's cat map, we show that the main result about the recurrent point of a linear vector field is not true for discrete-time flows.

The paper is structured as follows: In Section 2 we introduce all the background needed for the whole paper. We also prove some results concerning linear flows on vector spaces and some properties of linear vector fields admiting Jordan decomposition. In Section 3 we prove our main result characterizing the set of recurrent points of linear vector fields by means of its Jordan decomposition. Section 4 is used to prove that any linear vector field admits in fact a Jordan decomposition. We finish the paper with an Appendix where we analyze when the flow of a linear vector field is a flow of isometries for a given almost-Riemannian structure.

\subsection*{Notations}
If $H\subset G$ is a subgroup, we denote by $H_0$ the connected component of $H$ containing the identity element $e\in G$. If $H=\{e\}$ we say that $H$ is a trivial subgroup of $G$. By $L_g$ and $R_g$ we denote, respectively, the left and right-translations by $g$. The conjugation of $g$ is the map $C_g:=L_g\circ R_{g^{-1}}$. The center $Z(G)$ of $G$ is the set of elements in $G$ that satisfy $C_g=\id_G$. If $f:G\rightarrow H$ is a differentiable map between Lie groups, the differential of $f$ at $x$ is denoted by $(df)_x$. 	 

\section{Preliminaries}

Let $X$ be a topological space and $\{\phi_t\}_{t\in\R}$ a flow on $X$. For any $x\in X$ the {\it (positive) orbit} of $\phi_t$ is the set
$$\left(\OC^+(x, \phi):=\{\phi_t(x); \;\;t\geq 0\}\right)\;\;\;\OC(x, \phi):=\{\phi_t(x); \;\;t\in\R\}.$$

The set of {\it fixed points} of $\phi_t$ read as
$$\fix(\phi_t):=\{x\in X;\;\;\phi_t(x)=x, \;\;\forall t\in\R\}.$$

The set of {\it recurrent points} of $\phi_t$ is given by
$$\RC(\phi_t):=\left\{x\in X;\;\;\exists t_k\rightarrow+\infty;\;\;\phi_{t_k}(x)\rightarrow x\right\}.$$
Of course one have that $\fix(\phi_t)\subset\RC(\phi_t)$. In the next sections we relate the set of fixed and recurrent points with the Jordan decomposition of flows on vector spaces and more generally on Lie groups.

\subsection{Dynamics on vector spaces}

Let $V$ be a finite real vector space and $A:V\rightarrow V$ a linear map. Recall that $A$ is {\it semisimple} if its extension to the complexification $V_{\C}$ is diagonalizable. We say that $A$ is {\it nilpotent} if $A^n\equiv 0$ for some $n\in\N$ and $A$ is {\it elliptic (hyperbolic)} if it is semisimple and its eigenvalues are pure imaginary (real). The {\it (additive) Jordan decomposition} of $A$ reads as 
$$A=A_{\EC}+A_{\HC}+A_{\NC}, \;\;\mbox{ where }\;\;A, A_{\EC}, A_{\HC}\;\mbox{ and }\;A_{\NC}\;\mbox{ commute},$$
$A_{\EC}$ is elliptic, $A_{\HC}$ is hyperbolic and $A_{\NC}$ is nilpotent.

For any given eigenvalue $\lambda$ of $A_{\HC}$ let $V_{\lambda}:=\{v\in V;\;A_{\HC} v=\lambda v\}$ be its associated eigenspace. Define the $A, A_{\EC}, A_{\HC}, A_{\NC}$-invariant vector subspaces of $V$
$$V^+=\bigoplus_{\lambda>0}V_{\lambda}, \;\;\;V^0=\ker A_{\HC}\;\;\mbox{ and }\;\;V^-=\bigoplus_{\lambda<0}V_{\lambda}.$$
Then $V=V^+\oplus V^0\oplus V^-$ and for any norm on $V$ there exists $\lambda, t_0>0$ such that
\begin{equation}
\label{expanding}
\left\|\rme^{tA}_{|_{V^+}}\right\|\geq\rme^{t\lambda}\;\;\mbox{ and }\;\;\left\|\rme^{tA}_{|_{V^-}}\right\|\leq\rme^{-t\lambda}, \;\;\mbox{ for }\;t\geq t_0.
\end{equation}
Moreover, there is an inner product $\langle\cdot, \cdot\rangle$ in $V$ such that $\rme^{tA_{\EC}}$ is an isometry for any $t\in\R$.

We will use the notation $V^{+, 0}$ and $V^{-, 0}$ for the subspaces $V^+\oplus V^0$ and $V^-\oplus V^0$, respectively. 

The next lemma will be important in the context of linear vector fields ahead.

\begin{lemma}
	\label{gamma}
	Let $\gamma:\R\rightarrow V$ be a continuous curve that satisfies
	\begin{equation*}
	\gamma_{t+s}=\gamma_t+\rme^{tA}\gamma_s, \;\;\forall t, s\in \R.
	\end{equation*}
	If $A$ has no eliptical part and $(\gamma_{t_k})_{k\in\N}$ is bounded for some sequence $t_k\rightarrow\pm\infty$, then $\gamma_t\in V^{\mp}$ for all $t\in\R$. In particular, $(\gamma_t)_{t\geq 0}$ is bounded.
	\end{lemma}
	
	\begin{proof} By analogy it is only necessary to show that if $(\gamma_{t_k})_{k\in\N}$ is bounded for some sequence $t_k\rightarrow+\infty$ then $\gamma_t\in V^-$. Moreover, it is enough to consider the cases where $A$ is nilpotent or where it has only eigenvalues with positive real parts.
		
		In fact, if $W, U\subset V$ are $A$-invariant subspaces such that $V=W\oplus U$ and assume w.l.o.g. that $U=W^{\perp}$. If $\pi:V\rightarrow W$ is the orthogonal projection onto $W$ and we define $\gamma^W_t:=\pi(\gamma_t)$, the $A$-invariance of $W$ and $U$ implies that 
		$$\gamma^W_{t+s}=\pi(\gamma_{t+s})=\pi(\gamma_t+\rme^{tA}\gamma_s)=\pi(\gamma_t)+\pi(\rme^{tA}\gamma_s)=\pi(\gamma_t)+\rme^{t A}\pi(\gamma_s)=\gamma^W_t+\rme^{t A}\gamma^W_s.$$
		Moreover, the fact that $|\pi(v)|\leq |v|$ for any $v\in V$ gives us that 
		$$(\gamma_{t_k})_{k\in\N}\;\;\mbox{is bounded } \;\;\implies \;\;(\gamma^W_{t_k})_{k\in\N}\;\;\mbox{ is bounded.}$$
		In particular, the result follows if we show that 
		$$(\gamma^W_{t_k})_{k\in\N}\;\;\mbox{ is bounded} \;\;\implies \;\;\gamma^W\equiv 0\;\;\mbox{ for }\;\; W=V^+\;\;\mbox{ and }\; \;W=V^0,$$
		that is, we only have to analyze the cases where $A$ is nilpotent or it has only eigenvalues with positive real parts. 
				
		On the other hand, for any $k\in\N$ and $r\in (0, 1]$ we can write $t_k=rn_k+r_k$ with $n_k\in\N$ and $r_k\in [0, 1)$. Then
		$$|\gamma_{rn_k}|=|\gamma_{t_k-r_k}|=|\gamma_{-r_k}+\rme^{-r_kA}\gamma_{t_k}|\leq |\gamma_{-r_k}|+\|\rme^{-r_kA}\||\gamma_{t_k}|$$
		and 
		$$|\gamma_{t_k}|=|\gamma_{rn_k+r_k}|=|\gamma_{r_k}+\rme^{r_kA}\gamma_{rn_k}|\leq |\gamma_{r_k}|+\|\rme^{r_kA}\||\gamma_{rn_k}|.$$
		Since by continuity, 
		$$\{\gamma_{t}, \;\;t\in[-1, 1]\}\;\;\;\mbox{ and }\;\;\;\{\|\rme^{tA}\|, \;\;t\in [-1, 1]\}\;\;\;\mbox{ are bounded,}$$
		then		 
		$$(\gamma_{t_k})_{k\in\N}\;\;\mbox{ is bounded}\;\;\iff \;\;(\gamma_{rn_k})_{k\in\N}\;\;\mbox{ is bounded, }$$
		and so we can assume w.l.o.g. that $t_k=rn_k$ with $n_k\in \N$ and $r\in(0, 1]$. Moreover, since for any $t\in\R$ there exists $n\in\Z$ and $r\in(0, 1]$ such that $t=nr$, then 
		$$\gamma_r=0 \;\; \forall r\in(0, 1]\;\;\implies\;\;\gamma_t=\gamma_{nr}=\sum_{j=0}^n\rme^{jrA}\gamma_r=0\;\;\forall t\in \R$$ 
		and we only have to show that if $A$ is nilpotent or it has only eigenvalues with positive real parts then $(\gamma_{rn_k})_{k\in\N}$ bounded for some sequence $n_k\rightarrow+\infty$ implies $\gamma_r=0$. 
		
		\bigskip
		
		{\bf Case 1:} $A$ has only eigenvalues with positive real parts;
		
		Since,
		$$(1-\rme^{rA})\gamma_{rn_k}=(1-\rme^{rA})\sum_{j=0}^{n_k}\rme^{jrA}\gamma_r=(1-\rme^{(n_k+1)rA})\gamma_r,$$
		and $1-\rme^{rA}$ is invertible, it holds that 
		$$(\gamma_{rn_k})_{k\in\N}\;\;\mbox{ is bounded}\;\;\iff \;\;((1-\rme^{(n_k+1)rA})\gamma_r)_{k\in\N}\;\;\mbox{ is bounded. }$$
		However,  
		$$\left|(1-\rme^{(n_k+1)rA})\gamma_r\right|\geq \left|\rme^{(n_k+1)rA}\gamma_r\right|-\left|\gamma_r\right|\geq \rme^{(n_k+1)\lambda}\left|\gamma_r\right|-\left|\gamma_r\right|,$$
		and hence, 
		$$(\gamma_{rn_k})_{k\in\N}\;\;\mbox{ is bounded }\;\;\iff\;\; \gamma_r=0.$$
		
		\bigskip
		
		{\bf Case 2:} $A$ is nilpotent;
		
		Let $p\in\N$ be the greatest integer such that $A^p\gamma_r\neq 0$. Then
		$$\gamma_{rn_k}=\sum_{j=0}^{n_k}\rme^{jrA}\gamma_r=\sum_{j=0}^{n_k}\left(\gamma_r+jrA\gamma_r+\cdots+\frac{(jr)^p}{p!}A^p\gamma_r\right)=(n_k+1)\gamma_r+\left(\sum_{j=1}^{n_k}j\right)rA\gamma_r+\cdots+\left(\sum_{j=1}^{n_k}j^p\right)\frac{r^p}{p!}A^p\gamma_r.$$
		If $p\geq 1$, a simple calculation shows that 
		$$\lim_{k\rightarrow+\infty}\left(\sum_{j=1}^{n_k}j^p\right)^{-1}(n_k+1)=\lim_{k\rightarrow+\infty}\left(\sum_{j=1}^{n_k}j^p\right)^{-1}\left(\sum_{j=1}^{n_k}j^q\right)=0, \;\;\; \mbox{ if }\;\;1\leq q\leq p-1$$
		and hence, there exists $k_0\in\N$ such that 
		$$|\gamma_{rn_k}|\geq \frac{1}{2}\left(\sum_{j=1}^{n_k}j^p\right)\frac{r^p}{p!}\left|A^p\gamma_r\right|\geq n_k\frac{r^p}{2(p!)}\left|A^p\gamma_r\right|, \;\;k\geq k_0.$$
		Consequently, 
		$$(\gamma_{rn_k})_{k\in\N}\;\;\mbox{ bounded }\;\;\implies\;\; A^p\gamma_r=0, \;\;\forall p\geq 1.$$
		However, $A^p\gamma_r=0$ for all $p\geq 1$ implies necessarily that $\gamma_{rn_k}=(n_k+1)\gamma_r$ which is bounded if and only if $\gamma_r=0$.
		
		By the previous cases, $\gamma_t\in V^-$ for all $t>0$. By considering $0<t<s$ we have that 
		$$\gamma_{-t+s}=\gamma_{-t}+\rme^{-tA}\gamma_s\;\;\implies\;\;\gamma_{-t}=\gamma_{-t+s}-\rme^{-tA}\gamma_s\in V^-+\rme^{-tA}V^-\subset V^-,$$
		therefore $\gamma_t\in V^-$ for all $t\in\R$. 
		
		For the last assertion, let $t\geq t_1$ and consider $p_t\in\N$ and $r_t\in[0, 1)$ such that $t=p_tt_1+r_t$. Then,  
		$$\gamma_{t}=\gamma_{p_tt_1+r_t}=\gamma_{r_t}+\rme^{r_tA}\left(\sum_{j=0}^{p_t}\rme^{jt_1}\gamma_{t_1}\right),$$
		and so, by equation (\ref{expanding}), we obtain
		$$|\gamma_t|\leq |\gamma_{r_t}|+\rme^{-r_t\lambda}\left(\sum_{j=0}^{p_t}\rme^{-jt_1\lambda}\right)|\gamma_{t_1}|=|\gamma_{r_t}|+\rme^{-r_t\lambda}|\gamma_{t_1}|\frac{1-\rme^{-(p_t+1)t_1\lambda}}{1-\rme^{-t_1\lambda}}.$$
		Since 
		$$\rme^{-r_t\lambda}\left(1-\rme^{-(p_t+1)t_1\lambda}\right)\leq 1 \;\;\mbox{ and }\;\;\gamma\;\;\mbox{ is continuous, }$$
		the assertion follows.		
		\end{proof}	
		
		\bigskip
		
		We finish this section with some remarks concerning recurrent sets, fixed points and Jordan decomposition of the matrix exponential.
		
	\begin{remark}
		\label{a}
		A slight modification of Proposition 4.2 of \cite{Mau2} gives us that the recurrent set of $\rme^{tA}$ are related with the fixed points of $\rme^{tA_{\HC}}, \rme^{tA_{\NC}}$ and $\rme^{tA_{\HC, \NC}}$ by
		$$\RC(\rme^{tA})=\fix(\rme^{tA_{\HC}})\cap\fix(\rme^{tA_{\NC}})=\fix(\rme^{tA_{\HC, \NC}}),$$
		where $A_{\HC, \NC}:=A_{\HC}+A_{\NC}$.
	\end{remark}

	\begin{remark}
		\label{ad}
		If $V$ is a finite vector space and $A:V\rightarrow V$ is a linear map, the commutator of $A$ is defined by
		$$\ad(A):\mathfrak{gl}(V)\rightarrow\mathfrak{gl}(V), \;\;\ad(A)B:=[A, B]=BA-AB.$$
		Note that $[\ad(A), \ad(B)]=\ad([A, B])$ implying that $\ad(A)$ and $\ad(B)$ commutes when $A$ and $B$ commutes. Also, a simple calculation shows that 
		$$\ad(A)^k(B)=\sum_{j=0}^k(-1)^k\left(\begin{array}{c} k\\ j\end{array}\right) A^{k-j}BA^j$$
		and hence, $A^k=0$ implies $\ad(A)^{2k-1}=0$.
		Moreover, if $\{v_1, \ldots, v_m\}$ is a basis of $V$ such that $Av_i=\lambda_iv_i$, $i=1, \ldots, m$ and $E_{i, j}\in \mathfrak{gl}(V)$ satisfies 
		$$E_{i, j}v_k=\left\{\begin{array}{cc}
		v_i & \mbox{ if }k=j\\
		0 & \mbox{ if }k\neq j.
		\end{array}\right.\;\;\implies \ad(A)E_{i, j}=(\lambda_i-\lambda_j)E_{i, j},
		$$
		it follows that, $A$ diagonalizable implies $\ad(A)$ diagonalizable. Consequently, if $A=A_{\EC}+A_{\HC}+A_{\NC}$ is the Jordan decomposition of $A$ then $\ad(A)=\ad(A_{\EC})+\ad(A_{\HC})+\ad(A_{\NC})$ is the Jordan decomposition of $\ad(A)$.	
	\end{remark}

\subsection{Linear vector fields}

In this section we introduce the notion of  linear vector fields and its main properties. Let $G$ be a connected Lie group with Lie algebra $\mathfrak{g}$ identified with the set of left invariant vector fields. 

A vector field $\mathcal{X}$ on $G$ is said to be {\it linear} if its flow $\{\varphi_t\}_{t\in \R}$ is a $1$-parameter subgroup of $\mathrm{Aut}(G)$. Associated to any linear vector field $\mathcal{X}$ there is a derivation $\mathcal{D}$ of $\mathfrak{g}$ that satisfies
\begin{equation*}
(d\varphi_{t})_{e}=\mathrm{e}^{t\mathcal{D}}\; \; \; \mbox{ for all }\; \;
\;t\in \mathbb{R}. \label{derivativeonorigin}
\end{equation*}%
In particular, it holds that
\begin{equation*}
\varphi _{t}(\exp Y)=\exp (\mathrm{e}^{t\mathcal{D}}Y),\mbox{ for all }t\in 
\mathbb{R},Y\in \mathfrak{g}.
\end{equation*}

For any connected Lie group $G$ we denote by $T(G)$ the {\it toral component} of $G$, that is, the maximal compact, connected subgroup of $Z(G)_0$. As a standard fact we know that if $T(G)$ is trivial, then $Z(G)_0$ is simply connected and that $T\left(G/T(G)\right)$ is trivial (see \cite[Proposition 3.3]{Mau}).

Let $\{\varphi_t\}_{t\in\R}\subset\mathrm{Aut}(G)$ be a flow. By maximality, for any $t\in\R$ we get that $\varphi_t(T(G))=T(G)$. However, the group of the automorphisms of $T(G)$ is discrete and hence $\varphi_t|_{T(G)}=\id_{T(G)}$. In particular, it follows that 
$$T(G)\subset\fix(\varphi_t).$$

Let $\DC=\DC_{\EC}+\DC_{\HC}+\DC_{\NC}$ be the Jordan decomposition of $\DC$. By Theorem 3.2 and Proposition 3.3 of \cite{SM1} we have that $\DC_{\EC}, \DC_{\HC}\;\mbox{ and }\;\DC_{\NC}$ are still derivations of $\fg$. Moreover, if $\fg_{\lambda}$ stands for the eigenspaces of the hyperbolic part $\DC_{\HC}$, then $[\fg_{\lambda}, \fg_{\mu}]\subset\fg_{\lambda+\mu}$ if $\lambda+\mu$ is an eigenvalue of $\DC_{\HC}$ and zero otherwise (see \cite{SM1} Proposition 3.1). It turns out that 
$$\fg^+=\bigoplus_{\lambda>0}\fg_{\lambda}, \;\;\;\mbox{ and }\;\;\fg^-=\bigoplus_{\lambda<0}\fg_{\lambda}$$
are nilpotent Lie subalgebras. Denoting by $\fg^0=\ker\DC_{\HC}$ we obtain the decomposition $\fg=\fg^+\oplus\fg^0\oplus\fg^-$.

Now we extend the Jordan decomposition from a linear differential equation to linear vector fields. Let $\XC$ be a linear vector field on a connected Lie group $G$. We say $\XC$ is elliptic, hyperbolic or nilpotent, respectively, when its associated derivation $\DC$ is elliptic, hyperbolic or nilpotent. The Jordan decomposition of $\XC$ is given by 
$$\XC=\XC_{\EC}+\XC_{\HC}+\XC_{\NC}, \;\;\;\mbox{ where }\;\;\;\XC, \XC_{\EC}, \XC_{\HC}, \XC_{\NC}\;\;\mbox{ commute},$$
with $\XC_{\EC}$ elliptic, $\XC_{\HC}$ hyperbolic and $\XC_{\NC}$ nilpotent. By the uniqueness of the Jordan decomposition of $\DC$ and the connectedness of the group $G$ we get that the Jordan decomposition of $\XC$ is unique, when such decomposition exists. 

Let $G$ be a connected Lie group and $\XC$ a linear vector field with associated flow $\{\varphi_t\}_{t\in\R}$. If $\XC$ admits Jordan decomposition $\XC=\XC_{\EC}+\XC_{\HC}+\XC_{\NC}$ then 
$$\mbox{ for all }\;t\in\R,\;\varphi_t\;\mbox{ is a commutative product,}\;\;\varphi_t=\varphi_t^{\EC}\circ\varphi^{\HC}_t\circ\varphi^{\NC}_t,$$
where $\{\varphi^i_t\}_{t\in\R}$ is the flow of $\XC_i$ for $i=\EC, \HC, \NC$.  Moreover, for any $i, j\in\{\EC, \HC, \NC\}$ with $i\neq j$ we have that $\varphi_t^{i, j}:=\varphi_t^i\circ\varphi_t^j=\varphi_t^j\circ\varphi_t^i$ is the flow of the linear vector field $\XC_{i, j}:=\XC_i+\XC_j$ whose associated derivation is $\DC_{i, j}:=\DC_i+\DC_j$. 

\begin{remark}
	\label{conjugated}
	For any given linear vector field $\XC$ on $G$ and any surjective homomorphism $\pi:G\rightarrow H$ the family $\{\psi_t\}_{t\in\R}$ defined by the relation
	$$\pi\circ\varphi_t=\psi_t\circ\pi, \;\;\;t\in\R$$
	is a one-parameter group of automorphism of $H$ if and only if $\ker\pi$ is a $\varphi$-invariant subgroup. In particular, if $\YC$ stands for the linear vector field associated with $\{\psi_t\}_{t\in\R}$ then $\YC$ is elliptic, hyperbolic or nilpotent if $\XC$ is elliptic, hyperbolic or nilpotent, respectively.
\end{remark}

We define the {\it dynamical subgroups} of $G$ associated with the hyperbolical part of $\XC$ by
$$G^0=\fix(\varphi^{\HC}_t), \;\;G^+=\exp(\fg^+)\;\;\mbox{ and }\;\;G^-=\exp(\fg^-).$$

These subgroups where first defined in \cite{DS} without the use of Jordan decompositions. However, it is straightforward to see that their definitions are equivalent. The next proposition states the main properties of these dynamical subgroups. Its proof can be found at \cite[Proposition 2.9]{DS}.

\begin{proposition}
	\label{dynamical}
	It holds:
	\begin{enumerate}
		\item $G^0$ normalizes $G^+$ and $G^-$ and hence $G^{+, 0}:=G^+G^0$ and $G^{-, 0}:=G^-G^0$ are subgroups of $G$;
		\item $G^-\cap G^0=G^+\cap G^0=G^{-, 0}\cap G^+=G^{+, 0}\cap G^-=\{e\}$;
		\item The dynamical subgroups are closed in $G$; 
		\item If $G$ is solvable, then $G^0$ is connected and $G=G^-G^{+, 0}$. Moreover, $\fix(\varphi_t)\subset G^0.$
	\end{enumerate}
\end{proposition}
We will say that $G$ is {\it decomposable} if $G=G^-G^{+, 0}$. In particular, if $G$ is decomposable and $x\in G$ there exists unique $a\in G^-$, $b\in G^+$ and $c\in G^0$ such that $x=abc$.

A {\it Levi subgroup (subalgebra)} $S\subset G$ $(\fs\subset \fg) $ is a maximal connected semisimple subgroup (subalgebra). By Levi's Theorem, Levi subgroups (subalgebras) always exists and we have the Levi decomposition (see \cite[Chapter I-4]{ALEB})
$$G=SR, \;\;\;\mbox{ with }\;\;\dim(S\cap R)=0,$$
where $R$ is the solvable radical of $G$. 


\begin{proposition}
	\label{Levi}
	There exists a $\varphi^{\HC}$-invariant Levi subgroup $S\subset G$ such that 
	$$\fix(\varphi_t^{\HC})=\fix(\varphi_t^{\HC}|_S)\fix(\varphi_t^{\HC}|_R).$$
\end{proposition}

\begin{proof}
	Since $\{\rme^{t\DC_{\HC}}, \;t\in\R\}$ is a group of semisimple automorphisms of $\fg$,  Corollary 5.2 of \cite{Mostow} implies the existence of a $\DC_{\HC}$-invariant Levi subalgebra. Hence, the connected subgroup $S\subset G$ with Lie algebra $\fs$ is a Levi subgroup and the $\DC_{\HC}$-invariance of $\fs$ implies the $\varphi^{\HC}$-invariance of $S$.
	
	For the equality of the sets of fixed points, let us show the inclusion $\fix(\varphi_t^{\HC})\subset \fix(\varphi_t^{\HC}|_S)\fix(\varphi_t^{\HC}|_R)$ since the opposite one is trivial. Let $x\in\fix(\varphi_t^{\HC})$ and write it as $x=ab$ with $a\in S$, $b\in R$. Then,
	$$ab=x=\varphi_t(x)=\varphi_t^{\HC}(a)\varphi_t^{\HC}(b)\;\;\implies\;\; S\ni \varphi^{\HC}(a^{-1})a=\varphi^{\HC}_t(b)b^{-1}\in R.$$
	Since $\dim(R\cap S)=0$ and $t\mapsto \varphi^{\HC}_t(a^{-1})a$ is a continuous curve in $R\cap S$, we must necessarily have that $\varphi^{\HC}_t(a^{-1})a=\varphi^{\HC}_0(a^{-1})a=e$ implying that $a\in\fix(\varphi_t^{\HC}|_S)$ and also that $b\in\fix(\varphi_t^{\HC}|_{R})$ as desired. 
\end{proof}

\medskip
Let $\Delta\subset\fg$ be a subspace. A differentiable curve $\theta:[0, T]\rightarrow G$ is said to be {\it admissible} for $\Delta$ if  
$$\dot{\theta}(s)\in (dL_{\theta(s)})_e\Delta, \;\;\;\mbox{ almost everywhere.}$$

The next result implies that we can endow $G$ with a left-invariant Riemannian metric such that $\{\varphi_t^{\EC}\}_{t\in\R}$ is a flow of isometries.

\begin{theorem}
	\label{iso}
	Let $\Delta\subset\fg$ be a $\DC$-invariant subspace. If $\DC|_{\Delta}$ is skew-symmetric then $\{\varphi_t\}_{t\in\R}$ preserves the arc-lenght of any admissible curve for $\Delta$. In particular, $\{\varphi_t^{\EC}\}_{t\in\R}$ is a flow of isometries for some left-invariant metric on $G$.
\end{theorem}

\begin{proof}
	Let us denote by $\langle.,.\rangle$ the left-invariant metric and by $l(t)$ the length of the curve $s\longmapsto \varphi_t(\theta(s))$. Thanks to the left-invariance on the one hand, and to the fact that $(d\varphi_t)_e=\rme^{t\DC}$ is orthogonal on $\Delta$ for all $t$ on the other one, we get:
	$$\hspace{-9cm} l(t)=\int_{0}^{T} \left\langle \dds\varphi_t(\theta(s)),\dds\varphi_t(\theta(s)) \right\rangle^\frac{1}{2}_{\varphi_t(\theta(s))} \ ds$$
	$$\hspace{-7.2cm}= \int_{0}^{T}  \left\langle (d\varphi_t)_{\theta(s)}\dds \theta(s),(d\varphi_t)_{\theta(s)}\dds \theta(s)\right\rangle^\frac{1}{2}_{\varphi_t(\theta(s))} \ ds$$
	$$= \int_{0}^{T}  \left\langle (dL_{\varphi_t^e(\theta(s))})_e(d\varphi_t)_e(dL_{\theta(s)^{-1}})_{\theta(s)}\dds \theta(s),(dL_{\varphi_t^e(\theta(s))})_e(d\varphi_t)_e(dL_{\theta(s)^{-1}})_{\theta(s)}\dds \theta(s) \right\rangle^\frac{1}{2}_{\varphi_t(\theta(s))} \ ds$$
	$$\hspace{-9.2cm}=\int_{0}^{T}  \left\langle \dds \theta(s),\dds \theta(s) \right\rangle^\frac{1}{2}_{\theta(s)} \ ds = l(0),$$
	where for the third equality we used that $\varphi_t\circ L_g=L_{\varphi_t(g)}\circ\varphi_t$ for any $g\in G$ and $t\in\R$.
	
	The assertion on $\{\varphi_t^{\EC}\}_{t\in\R}$ follows from the fact that there exists an inner product $\langle\cdot, \cdot \rangle$ on $\fg$ such that $\{\rme^{t\DC_{\EC}}\}_{t\in\R}$ is a flow of isometries, or equivalently, $\DC_{\EC}$ is skew-symmetric. By extending $\langle\cdot, \cdot \rangle$ to a left-invariant metric on $G$ we have that $\{\varphi^{\EC}_t\}$ preserves the arc-lenght of any curve in $G$ and the result follows.
\end{proof}

\begin{proposition}
	\label{aff}
	It holds:
	\begin{enumerate}
		\item The (positive) orbit of $\varphi_t$ at $x\in G$ is bounded iff the (positive) orbit of $\varphi_t^{\HC, \NC}$ at $x\in G$ is bounded;
		\item If $x\in\RC(\varphi_t)$ then  $\cl(\OC(x, \varphi))=\cl(\OC^+(x, \varphi)).$
	\end{enumerate}
\end{proposition}

\begin{proof}
	1. Let us denote by $\varrho$ the distance associated with the left-invariant metric such that $\{\varphi_t^{\EC}\}$ is a flow of isometries. Then, 
	$$\mbox{ for all }t\in\R, \;\;\varrho(\varphi^{\HC, \NC}_t(x), e)=\varrho(\varphi^{\EC}_t(\varphi_t^{\HC, \NC}(x)), \varphi_t^{\EC}(e))=\varrho(\varphi_t(x), e),$$
	which implies the result.
	
	2. Since $\OC^+(x, \varphi)\subset\OC(x, \varphi)$ it is enough to show that $\varphi_{-t}(x)\in\cl(\OC^+(x, \varphi))$ for any $t>0$. However, since $x\in\RC(\varphi_t)$ there exists $t_k\rightarrow+\infty$ such that $\varphi_{t_k}(x)\rightarrow x$ as $k\rightarrow+\infty$ and so
	$$\varphi_{t_k-t}(x)=\varphi_{-t}(\varphi_{t_k}(x))\rightarrow\varphi_{-t}(x), \;\;k\rightarrow+\infty.$$
	Since $t_k\rightarrow+\infty$, there exists $k_0\in\N$ such that $t_k-t>0$ for any $k\geq k_0$ and so $\varphi_{t_k-t}(x)\in\OC^+(x, \varphi)$ implying that $\varphi_{-t}(x)\in\cl(\OC^+(x, \varphi))$ as stated.
\end{proof}


\section{Recurrent points of linear vector fields}

Our aim in this section is to prove the following fact: If a linear vector field admits a Jordan decomposition, the set of recurrent points coincides with the intersection of the sets of fixed points of the hyperbolic and nilpotent parts of the mentioned decomposition. 

Let us fix a linear vector field $\XC$ and assume through the whole section that $\XC$ admits a Jordan decomposition. We start by showing the result in question on the adjoint group.

\begin{proposition}
	\label{previousLemma}
	Let $\XC$ be a linear vector field on $G$ with associated derivation $\DC$. Then,
	\begin{equation}
	\label{Ad}
	\Ad(\varphi_t(g))=\rme^{t\ad(\DC)}\Ad(g), \;\;g\in G, t\in\R.
	\end{equation}
	In particular,
	$$\RC\left(\Ad\circ\varphi_t\right)=\fix\left(\Ad\circ\varphi_t^{\HC, \NC}\right)=\fix\left(\Ad\circ\varphi_t^{\HC}\right)\cap\fix\left(\Ad\circ\varphi_t^{\NC}\right),$$
	where $\Ad\circ\varphi_t(\Ad(g)):=\Ad(\varphi_t(g))$.

\end{proposition}

\begin{proof}
	Since $G$ is connected, we only need to show equation (\ref{Ad}) for $g=\rme^X$ with $X\in\fg$. However, 	
	$$\Ad(\varphi_t(\exp(X)))=\Ad(\exp(\rme^{t\DC}X))=\rme^{\ad\left(\rme^{t\DC}X\right)}=\rme^{\rme^{t\DC}\ad(X)\rme^{-t\DC}}=\rme^{t\DC}\rme^{\ad(X)}\rme^{-t\DC}=C_{\rme^{t\DC}}\Ad(\exp(X)).$$
	Furthermore, since $C_{\rme^{t\DC}}:\mathfrak{gl}(\fg)\rightarrow \mathfrak{gl}(\fg)$ is linear, it holds that 
	$$C_{\rme^{t\DC}}=\Ad(\rme^{t\DC})=\rme^{t\ad(\DC)}\;\;\implies\;\;\Ad(\varphi_t(\exp(X)))=\rme^{t\ad(\DC)}\Ad(\exp(X)),$$
	which proves equation (\ref{Ad}).
	
	By equation (\ref{Ad}) it turns out that $\Ad\circ\varphi_t=\rme^{t\DC}|_{\Ad(G)}$ implying 
	$$\RC\left(\Ad\circ\varphi_t\right)\subset\RC\left(\rme^{t\ad(\DC)}\right)\cap\Ad(G).$$
	On the other hand, $\ad(\DC)$ is a linear map of $\mathfrak{gl}(\fg)$ with Jordan decomposition $\ad(\DC)=\ad(\DC_{\EC})+\ad(\DC_{\HC})+\ad(\DC_{\NC})$ which by Remarks \ref{a} and  \ref{ad} implies that
	$$\RC\left(\rme^{t\ad(\DC)}\right)=\fix\left(\rme^{t\ad(\DC_{\HC, \NC})}\right)=\fix\left(\rme^{t\ad(\DC_{\HC})}\right)\cap\fix\left(\rme^{\ad(\DC_{\NC})}\right).$$ 
	Since 
	$$\fix\left(\rme^{t\ad(\DC_{\HC, \NC})}\right)\cap\Ad(G)=\fix\left(\Ad\circ\varphi_t^{\HC, \NC}\right) \;\;\mbox{ and }\;\;$$
	$$\fix\left(\rme^{t\DC_{\HC}}\right)\cap\fix\left(\rme^{t\DC_{\NC}}\right)\cap\Ad(G)=\fix\left(\Ad\circ\varphi_t^{\HC}\right)\cap\fix\left(\Ad\circ\varphi_t^{\NC}\right),$$
	we get that 
	$$\RC\left(\Ad\circ\varphi_t\right)\subset \fix\left(\Ad\circ\varphi_t^{\HC, \NC}\right)=\fix\left(\Ad\circ\varphi_t^{\HC}\right)\cap\fix\left(\Ad\circ\varphi_t^{\NC}\right)\subset \RC\left(\Ad\circ\varphi_t\right),$$
	where the last inclusion follows from the fact that by Theorem \ref{iso}, $\Ad\circ\varphi^{\mathcal{E}}=\rme^{t\ad(\DC_{\mathcal{E}})}$ is an isometry, for any $t\in \R$.
\end{proof}

\bigskip

Furthermore, it is possible to prove something more

\begin{proposition}
	\label{fixx}
	With the previous notations, it holds that 
	$$\fix(\varphi_t^{\HC, \NC})=\fix(\varphi_t^{\HC})\cap\fix(\varphi_t^{\NC}).$$
\end{proposition}

\begin{proof} Our proof consider the different classes of Lie groups as follows: 
	
	$\bullet$\; Assume first that $G$ is a semisimple Lie group;
	
	By Proposition \ref{previousLemma} it follows that 
	$$x\in \fix(\varphi^{\HC, \NC})\;\;\implies\;\;\forall t\in\R, \;\; x^{-1}\varphi^{i}_t(x)\in Z(G), \; i=\HC, \NC.$$
	Since the center of a semisimple Lie group is discrete we must have $x^{-1}\varphi^{i}_t=e$ for all $t\in\R$. Therefore,  
	$$\fix(\varphi^{\HC, \NC})\subset\fix(\varphi^{\HC})\cap\fix(\varphi^{\NC}),$$ 
	which implies the equality.
	
	$\bullet$\; Let us assume now that $G$ is a solvable Lie group;
	
	Since the corresponding dynamical subgroups of $\varphi^{\HC, \NC}$ and $\varphi^{\HC}$ coincides, we have by Proposition \ref{dynamical} item 4. that 
	$$\fix(\varphi^{\HC, \NC})\subset G^0=\fix(\varphi^{\HC})\;\;\;\implies\;\;\;\fix(\varphi^{\HC, \NC})\subset \fix(\varphi^{\HC})\cap\fix(\varphi^{\NC}).$$
	
	$\bullet$\; Finally, let $G$ be an arbitrary Lie group.
	
	Let $S$ be a $\varphi^{\HC}$-invariant Levi subgroup given by Proposition \ref{Levi} and consider $\pi:G\rightarrow G/R$ be the canonical projection. Denote by $\bar{\varphi}^{\HC, \NC}, \bar{\varphi}^{\HC}$ and $\bar{\varphi}^{\HC}$ the flows induced by $\varphi^{\HC, \NC}, \varphi^{\HC}$ and $\varphi^{\NC}$ on $G/R$, respectively. Since $G/R$ is semisimple, we obtain that
	$$\pi(\fix(\varphi_t^{\HC, \NC}))\subset \fix(\bar{\varphi}_t^{\HC, \NC})= \fix(\bar{\varphi}_t^{\HC})\cap\fix(\bar{\varphi}_t^{\NC}).$$
	Let then $x\in\fix(\varphi^{\HC, \NC}_t)$ and write it as $x=ab$ with $a\in S$ and $b\in R$. By the $\varphi^{\HC}$-invariance of $S$ and $R$ we get $$\varphi^{\HC}_t(x)=\varphi^{\HC}_t(ab)=\varphi_t^{\HC}(a)\varphi^{\HC}_t(b)\;\;\implies\;\;\pi(\varphi^{\HC}_t(x))=\pi(\varphi^{\HC}(a))$$
	and hence
	$$\pi(a)=\pi(x)=\bar{\varphi}_t^{\HC}(\pi(x))=\pi(\varphi_t^{\HC}(x))=\pi(\varphi^{\HC}_t(a))\;\;\implies\;\;S\ni a^{-1}\varphi_t^{\HC}(a)\in R.$$
	Since $\dim(R\cap S)=0$ and $t\mapsto a^{-1}\varphi_t^{\HC}(a)$ is a continuous curve, it follows that $a\in\fix(\varphi^{\HC}_t)$. On the other hand, 
	$$\fix(\varphi_t^{\HC, \NC})_0=\left(\fix(\varphi_t^{\HC})\cap\fix(\varphi^{\HC}_t)\right)_0\;\;\implies\;\;x^{-1}\varphi_t^{\HC}(x)\in \fix(\varphi_t^{\HC})\;\;\implies\;\;b^{-1}\varphi^{\HC}_t(b)\in \fix(\varphi_t^{\HC})\cap R\subset R^0.$$
	Let us write $\varphi_t^{\HC}(b)=bc_t$ with $c_t\in R^0$. Since $b\in R$ and $R$ is solvable, Proposition \ref{dynamical} implies the existence of unique $p\in R^-$, $q\in R^+$ and $r\in R^0$ such that $b=pqr$. Hence, 
	$$\varphi_t^{\HC}(pqr)=pqrc_t\;\;\;\implies\;\;\; \varphi^{\HC}_t(p)=p, \;\; \varphi^{\HC}_t(q)=q, \;\;\varphi^{\HC}_t(r)=rc_t.$$
	However, $\fix(\varphi_t^{\HC}|_{R})=R^0$ implying that $p=q=e$ and consequently that $b=r$. In particular, we get that $b\in R^0=\fix(\varphi^{\HC}_t|_R)$ and by Proposition \ref{Levi} we conclude that 
	$$x=ab\in\fix(\varphi_t^{\HC}|_S)\fix(\varphi_t^{\HC}|_R)=\fix(\varphi_t^{\HC}).$$ 
	Since $\fix(\varphi_t^{\HC, \NC})\cap\fix(\varphi^{\HC}_t)=\fix(\varphi^{\NC}_t)$ we get that $x\in\fix(\varphi^{\HC})\cap\fix(\varphi^{\NC}_t)$ implying the equality
	$$\fix(\varphi_t^{\HC, \NC})=\fix(\varphi^{\HC})\cap\fix(\varphi^{\NC}_t),$$
	and concluding the proof.
\end{proof}

\bigskip

Next, we prove our main result concerning recurrent points. We build the proof by considering the cases where toral component of $G$ is trivial or not.

\begin{theorem}
	\label{fixedpoints}
	With the previous notations, it holds that
	$$\RC(\varphi_t)=\fix(\varphi_t^{\HC, \NC})=\fix(\varphi_t^{\HC})\cap\fix(\varphi_t^{\NC}).$$
\end{theorem}

\begin{proof} By Proposition \ref{fixx} we just need to show that 
	$$\RC(\varphi_t)=\fix(\varphi_t^{\HC, \NC}).$$ 
	
	Let us consider $x\in \fix(\varphi_t^{\HC, \NC})$ and let $\varrho$ be the left-invariant Riemannian metric for such that $\{\varphi_t^{\EC}\}_{t\in\R}$ is a flow of isometries. Therefore,  
	$$\varrho(\varphi_t(x), e)=\varrho(\varphi^{\EC}_t(x), e)=\varrho(\varphi^{\EC}_t(x), \varphi^{\EC}_t(e))=\varrho(x, e), \;\;\forall t\in\R.$$
	Thus, there exists $t_k\rightarrow+\infty$ such that $(\varphi^{\EC}_{t_k}(x))_{k\in\N}$ is convergent. In particular, for any $\varepsilon>0$ there exists $k_0\in\N$ such that $t_k> t_{k_0}$ implies 
	$$\varrho(\varphi^{\EC}_{t_{k}}(x), \varphi^{\EC}_{t_{k_0}}(x))<\varepsilon\;\;\implies\;\;\varrho(\varphi_{t_k-t_{k_0}}(x), x)=\varrho(\varphi^{\EC}_{t_k-t_{k_0}}(x), x)=\varrho(\varphi^{\EC}_{t_{k}}(x), \varphi^{\EC}_{t_{k_0}}(x))<\varepsilon,$$
	showing that $x\in\RC(\varphi_t)$ and hence $\fix(\varphi_t^{\HC, \NC})\subset\RC(\varphi_t)$.  
	
	Let $x\in\RC(\varphi_t)$. Since $\Ad(\RC(\varphi_t))\subset \RC(\Ad\circ\varphi_t)$, Proposition \ref{previousLemma} shows that   
	$$\forall t\in\R, \;\;\Ad(\varphi^{\HC, \NC}_t(x))=\Ad(x)\;\;\mbox{ or equivalently } \;\;\forall t\in\R, \;\;x^{-1}\varphi^{\HC, \NC}_t(x)\in Z(G).$$
	Let us define the curve
	$$\zeta:\R\rightarrow Z(G), \;\;\; t\in \R\mapsto \zeta_t:=x^{-1}\varphi^{\HC, \NC}_t(x).$$ 
	In particular,  
	$$x\in\fix(\varphi_t^{\HC, \NC})\;\;\iff\;\;\zeta\equiv e.$$
	Moreover, since $\zeta$ is continuous and $\zeta_0=e$ it follows that $\zeta_t\in Z(G)_0$ for all $t\in\R$ and we can divide our analysis in two cases:
	
	{\bf Case 1:} $T(G)$ is trivial;
	
	Here, $Z(G)_0$ is simply connected and the exponential map $\exp:\fz(\fg)\rightarrow Z(G)_0$ is a diffeomorphism. Therefore, the curve, 
	$$\gamma:\R\rightarrow\fz(\fg)\;\;\mbox{ given by the formula }\;\;\zeta_t=\exp\gamma_t,$$ 
	is well defined. Moreover, for any $t, s\in\R$ obtain
	$$\zeta_{t+s}=x^{-1}\varphi^{\HC, \NC}_{t+s}(x)=x^{-1}\varphi_t^{\HC, \NC}\left(\varphi^{\HC, \NC}_s(x)\right)=\left(x^{-1}\varphi^{\HC, \NC}_t(x)\right)\left(\varphi^{\HC, \NC}_t\left(x^{-1}\varphi^{\HC, \NC}_s(x)\right)\right)=\zeta_t\varphi^{\HC, \NC}_t(\zeta_s)$$
	which implies 
	$$\exp\gamma_{t+s}=\exp\gamma_t\varphi_t^{\HC, \NC}\left(\exp\gamma_s\right)=\exp\gamma_t\exp\left(\rme^{t\DC_{\HC, \NC}}\gamma_s\right)=\exp\left(\gamma_t+\rme^{t\DC_{\HC, \NC}}\gamma_s\right).$$
	Therefore, 
	$$\forall t, s\in\R, \;\;\;\;\;\gamma_{t+s}=\gamma_t+\rme^{t\DC_{\HC, \NC}}\gamma_s.$$
	On the other hand, there exists $t_k\rightarrow+\infty$ such that 
	$$\varphi_{t_k}(x) \rightarrow x\;\;\implies\;\; (\varphi^{\HC, \NC}_{t_k}(x))_{k\in\N}\;\mbox{ is bounded }\;\;\implies\;\; (\zeta_{t_k})_{k\in\N}\;\;\mbox{ is bounded }\;\;\implies\;\;(\gamma_{t_k})_{k\in\N}\;\;\mbox{ is bounded }$$
	where the first implication follows from the fact that $\varphi_t^{\EC}$ is an isometry for every $t\in\R$. Since $\DC_{\HC, \NC}$ has no eliptical part, Lemma \ref{gamma} implies that $\gamma_t\in \fg^-$ for all $t\in\R$ and that $(\gamma_t)_{t\geq 0}$ is bounded. Then,
	$$\zeta_t=\exp\gamma_t\in G^-, \;\;\forall t\in\R \;\;\mbox{ and }\;\;(\zeta_t)_{t\geq 0}\;\;\mbox{ is bounded.}$$
	Since $\zeta_t=x^{-1}\varphi_t^{\HC, \NC}(x)$ we conclude that $\OC^+(x, \varphi^{\HC, \NC})$ is also bounded. By Proposition \ref{aff} and the fact that $x\in\RC(\varphi_t)$ we conclude that 
	$$\cl(\OC^+(x, \varphi^{\HC, \NC}))\;\;\mbox{ is bounded }\;\;\implies\;\;\cl(\OC^+(x, \varphi))=\cl(\OC(x, \varphi))\;\;\mbox{ is bounded }\;\;\implies\;\;\cl(\OC(x, \varphi^{\HC, \NC}))\;\;\mbox{ is bounded}.$$
	Therefore, $(\zeta_t)_{t\in\R}$ is bounded and by applying Lemma \ref{gamma} again we concluded that $\zeta_t\in G^+$ for all $t\in\R$. Therefore, 
	$$\forall t\in\R, \;\;\zeta_t\in G^+\cap G^-=\{e\}\;\;\implies \;\;x\in\fix(\varphi_t^{\HC, \NC}).$$
	
	\bigskip
	
	{\bf Case 2:} $T(G)$ is nontrivial.
	
	Since $G/T(G)$ has trivial toral component, we get by the previous case that 
	$$\RC(\varphi_t)\subset \fix(\varphi^{\HC, \NC}_t)T(G).$$ 
	However, $T(G)\subset\fix(\varphi^{\HC, \NC})$ and hence $\RC(\varphi_t)\subset \fix(\varphi^{\HC, \NC}_t)$, finishing the proof.
\end{proof}

\bigskip

As a direct consequence we have the following

\begin{corollary}
	\label{fix}
	With the previous notations, it holds that 
	$$\fix(\varphi_t)=\fix(\varphi^{\EC}_t)\cap\fix(\varphi^{\HC}_t)\cap\fix(\varphi^{\NC}_t).$$
\end{corollary}

\begin{proof}
	Since any fixed point is recurrent, Theorem \ref{fixedpoints} implies that 
	$\fix(\varphi_t)\subset\RC(\varphi_t)=\fix(\varphi^{\HC}_t)\cap\fix(\varphi^{\NC}_t).$
	Therefore, if $x\in\fix(\varphi_t)$ we get $x=\varphi_t(x)=\varphi_t^{\EC}(\varphi^{\HC, \NC}_t(x))=\varphi_t^{\EC}(x)$ implying that 
	$$\fix(\varphi_t)\subset\fix(\varphi^{\EC}_t)\cap\fix(\varphi^{\HC}_t)\cap\fix(\varphi^{\NC}_t).$$
	Since the opposite inclusion is trivial, the result follows.
\end{proof}

\begin{remark}
	It is important to remark that Theorem \ref{decomposition} was proved in \cite[Lemma 4.3]{Mau2} in the context of automorphisms on simply connected connected Lie groups.
\end{remark}

\bigskip
The next example show that Theorem \ref{fixedpoints} does not holds for discrete-time flows.

\begin{example}
	The $2$-dimensional torus $\T^2=\R^2/\Z^2$ is an abelian Lie group whose Lie algebra is $\R^2$. Let 
	$$A=\left(\begin{array}{cc}
	\ln\left(\frac{3+\sqrt{5}}{2}\right)  & 0\\ 0  &  \ln\left(\frac{3-\sqrt{5}}{2}\right)
	\end{array}\right)\;\;\;\;\;\mbox{ and }\;\;\;\;\; P=\left(\begin{array}{cc}
	2 & 2\\ 1+\sqrt{5}  &  1-\sqrt{5}
	\end{array}\right)$$
	and consider the hyperbolic derivation $\DC=PAP^{-1}$. Since $\tr \DC=0$ we have that 
	$$\rme^{\DC}=\left(\begin{array}{cc}
	1 & 1\\ 1  &  2
	\end{array}\right),$$
	induces the automorphism $\varphi\in\mathrm{Aut}(\T^2)$ given by 
	$$\varphi \left([x, y]\right)=[x+y, x+2y], \;\;\;\mbox{ where }\;\;[x, y]:=(x, y)+\Z^2.$$
	Such automorphism, known as {\it Arnold's cat map} is the recurrent in the torus and hence $\RC(\varphi_n)=\T^2$ (see \cite[Example 1.16]{Arn}), that is $\RC(\psi_n)=\T^2$, where $\varphi_n$ is the discrete-time flow 
	$$\varphi_n:=\underbrace{\varphi\circ\cdots\circ\varphi}_{n-times}.$$
	On the other hand, a simple calculation shows that $\fix(\varphi_n)=\{[0, 0]\}$.
\end{example}

\section{Jordan decomposition}

As a relevant consequence of the main result in Theorem \ref{fixedpoints}, in this section we prove that any linear vector field $\XC$ admits a Jordan decomposition. In particular, Theorem \ref{fixedpoints} holds for any linear vector field on connected groups.

\begin{theorem}
	\label{decomposition}
	Any linear vector field $\XC$ on a connected Lie group $G$ admits a Jordan decomposition.
\end{theorem}

\begin{proof}
	Let us first assume that $G$ is simply connected and let $\DC=\DC_{\EC}+\DC_{\HC}+\DC_{\NC}$ be the Jordan decomposition of $\DC$. Due to the extra topological assumption on $G$, Theorem 2.2 of \cite{VAJT} assures the existence of linear vector fields $\XC_i$ associated with the derivations $\DC_i$ for $i=\EC, \HC, \NC$. By definition, $\XC_{\EC}$ is elliptic, $\XC_{\HC}$ is hyperbolic and $\XC_{\NC}$ is nilpotent.
	
	\begin{enumerate}
		\item Let us start by showing that $\XC=\XC^{\EC}+\XC^{\HC}+\XC^{\NC}$. If $\{\varphi_t^{\EC}\}_{t\in\R}, \{\varphi_t^{\HC}\}_{t\in\R}$ and $\{\varphi_t^{\NC}\}_{t\in\R}$ are the flows of $\XC_{\EC}, \XC_{\HC}$ and $\XC_{\NC}$, respectively, it follows that
		$$\varphi_t(\exp X)=\exp{\rme^{t\DC}X}=\exp(\rme^{t(\DC_{\EC}+\DC_{\HC}+\DC_{\NC})}X)= \exp(\rme^{t\DC_{\EC}}\rme^{t\DC_{\HC}}\rme^{t\DC_{\NC}}X)$$
		$$=\varphi_t^{\EC}(\exp(\rme^{t\DC_{\HC}}\rme^{t\DC_{\NC}}X))=\varphi_t^{\EC}(\varphi_t^{\HC}(\exp(\rme^{t\DC_{\NC}}X)))=\varphi_t^{\EC}(\varphi_t^{\HC}(\varphi_t^{\NC}(\exp X))).$$
	    By the connectedness of $G$ we conclude that
		$$\varphi_t=\varphi_t^{\EC}\circ\varphi_t^{\HC}\circ\varphi_t^{\NC}, \;\;t\in\R.$$
		Finally, by derivation we get that $\XC=\XC_{\EC}+\XC_{\HC}+\XC_{\NC}$ as stated.
		
		\item  Next, we prove that the vector fields $\XC, \XC_{\EC}, \XC_{\HC}$ and $\XC_{\NC}$ commutes. We just show that $[\XC, \XC_{\EC}]=0$, since the other cases are analogous. However, $[\XC, \XC_{\EC}]=0$ is equivalent to $\varphi_s\circ\varphi_t^{\EC}=\varphi_t^{\EC}\circ\varphi_s$, $t, s\in\R$, which follows from the commutativeness of $\DC$ and $\DC_{\EC}$. In fact, for any $X\in\fg$ it holds that 
		$$\varphi_s(\varphi_t^{\EC}(\exp X))=\varphi_s(\exp(\rme^{t\DC_{\EC}}))=\exp(\rme^{s\DC}\rme^{t\DC_{\EC}}X)$$
		$$=	\exp(\rme^{t\DC_{\EC}}\rme^{s\DC}X)=\varphi_t^{\EC}(\exp(\rme^{s\DC}))=\varphi_t^{\EC}(\varphi_s(\exp X)).$$
		The connectedness of $G$ implies $\varphi_s\circ\varphi_t^{\EC}=\varphi_t^{\EC}\circ\varphi_s$, for any $t, s\in\R$, as claimed.
	\end{enumerate}
	
	By 1. and 2. we obtain that any linear vector field on a connected simply connected Lie group admits Jordan decomposition. Let $G$ be a connected Lie group and denote by $\widetilde{G}$ the connected, simply connected covering of $G$. If $\DC$ is the derivation associated with $\XC$, Theorem 2.2 of \cite{VAJT} assures the existence of a linear vector field $\widetilde{\XC}$ on $\widetilde{G}$ with associated derivation $\DC$. Again, the connectedness of $\widetilde{G}$ gives us
	\begin{equation}
	\label{flow}
	\mbox{ for all }\;\;t\in\R, \;\;\pi\circ\widetilde{\varphi}_t=\varphi_t\circ\pi,
	\end{equation}
	where $\pi:\widetilde{G}\rightarrow G=\widetilde{G}/D$ is the canonical projection, $D\subset Z(G)$ is a discrete subgroup and  $\{\varphi_t\}_{t\in\R}$, $\{\widetilde{\varphi}_t\}_{t\in\R}$ denote the flows of $\XC$, $\widetilde{\XC}$, respectively. Equation (\ref{flow}) implies in particular that $\widetilde{\varphi}_t(\ker\pi)=\ker\pi$ and since $\ker\pi$ is discrete we obtain by continuity that $\ker\pi\subset\fix(\widetilde{\varphi}_t)$.
	
	By the simply connected case, $\widetilde{\XC}$ admits a Jordan decomposition $\widetilde{\XC}=\widetilde{\XC}_{\EC}+\widetilde{\XC}_{\HC}+\widetilde{\XC}_{\NC}$. If we denote by $\{\widetilde{\varphi}^i_t\}_{t\in\R}$ the flow of $\widetilde{\XC}_i$ for $i=\EC, \HC, \NC$, we get by Corollary \ref{fix} that 
	$$\fix(\widetilde{\varphi}_t)=\fix(\widetilde{\varphi}^{\EC}_t)\cap\fix(\widetilde{\varphi}^{\HC}_t)\cap\fix(\widetilde{\varphi}^{\NC}_t)$$ 
	and hence $\ker\pi\subset \fix(\widetilde{\varphi}^i_t)$ for $i=\EC, \HC, \NC$. In particular, 
	$$\varphi_t^i:G\rightarrow G, \;\mbox{ defined by the relation }\;\;\pi\circ\widetilde{\varphi}^i_t=\varphi^i_t\circ\pi,\;\;\;t\in\R, \;i=\EC, \HC, \NC$$
	is a one-parameter subgroup of automorphisms of $G$ and its associated vector field $\XC_i$ is elliptic, hyperbolic or nilpotent if $\widetilde{\XC}_i$ is elliptic, hyperbolic or nilpotent (see Remark \ref{conjugated}). Therefore, $\XC=\XC_{\EC}+\XC_{\HC}+\XC_{\NC}$ is the Jordan decomposition of $\XC$ proving the result for the connected case and finishing the proof.	
\end{proof}

\bigskip

As a consequence of Theorems \ref{fixedpoints} and \ref{decomposition} we get the following result.

\begin{theorem}
	The recurrent set of the flow of a linear vector field $\XC$ on a connected Lie group $G$ is
	given as the intersection of the fixed points of the hyperbolic and nilpotent components of its Jordan
	decomposition.
\end{theorem}

\begin{remark}
	It is important to remark that Theorem \ref{fixedpoints} was proved first in \cite[Theorem 4.4]{Mau2} in the context of automorphisms on simply connected, connected nilpontent Lie groups.
\end{remark}

\begin{example}
	Let $\theta=\left(\begin{array}{cc}
	0  & -1\\ 1  &  0
	\end{array}\right)$ and define $\rho_t:=\rme^{t\theta}$. Consider the semi-direct product $G=\R\times_{\rho}\R^2$. Following \cite{DSAy} a linear vector field of $G$ and its associated derivation are given, respectively, by 
	$$\XC(t, v)=(0, Av+\Lambda_t\xi),	\;\;\;\;\;\mbox{ and }\;\;\DC=\left(\begin{array}{cc}
		0  &  0\\ \xi  &  A
	\end{array}\right),$$
	where  $\xi\in\R^2$, $\theta=\left(\begin{array}{cc}
	\lambda  & -\mu\\ \mu  &  \lambda
	\end{array}\right)$ and $\Lambda_t=(\rho_t-1)\theta^{-1}$. If $\lambda^2+\mu^2\neq 0$, the the eliptic, hyperbolic and nilpotent parts of $A$ reads, respectively, as
	$$A_{\EC}=\left(\begin{array}{cc}
	0  & -\mu\\ \mu  &  0
	\end{array}\right), \;\;\; A_{\HC}=\left(\begin{array}{cc}
		\lambda  & 0\\ 0  &  \lambda
	\end{array}\right)\;\;\;\mbox{ and }\;\;\; A_{\NC}=0.$$
	Consequently, 	   
	$$\DC_{\EC}=\left(\begin{array}{cc}
	0  &  0\\ A_{\EC}A^{-1}\xi  &  A_{\EC}
	\end{array}\right),\;\;\;\;\DC_{\HC}=\left(\begin{array}{cc}
	0  &  0\\ A_{\HC}A^{-1}\xi  &  A_{\HC}
	\end{array}\right)\;\;\mbox{ and }\;\;\DC_{\NC}=0.$$
	are the corresponding eliptic, hyperbolic and nilpotent parts of $\DC$. Hence,
	$$\XC_{\EC}(t, v)=(0, A_{\EC}v+\Lambda_tA_{\EC}A^{-1}\xi), \;\;\; \XC_{\HC}(t, v)=(0, A_{\HC}v+\Lambda_tA_{\HC}A^{-1}\xi)\;\;\mbox{ and }\;\;\XC_{\NC}=0.$$
	Since $\XC_{\NC}=0$ we get 
	$$\RC(\varphi_t)=\fix(\varphi_t^{\HC})=\{(t, v)\in G; \;\;\XC_{\HC}(t, v)=(0, 0)\}.$$
	However, a simple calculation shows us that
	$$\{(t, v)\in G; \;\;\XC_{\HC}(t, v)=(0, 0)\}=\{(t, v)\in G; \;\;\XC(t, v)=(0, 0)\},$$
	implying that 
	$$\RC(\varphi_t)=\fix(\varphi_t)=\{(t, v)\in G; \;\; Av+\Lambda_t\xi=0\}.$$
	In particular, if $\xi=0$ we get $\RC(\varphi_t)=\R\times\{0\}$.
\end{example}

\appendix

\section{One-parameter groups of isometries of left-invariant sub-Riemannian metrics}

In what follows $G$ is a $n$-dimensional connected Lie group and $\fg$ stands for its Lie algebra. The group is endowed with a sub-Riemannian, possibly Riemannian, left-invariant structure defined by a set $\{Y_1,Y_2,\dots,Y_p\}$ of left-invariant vector-fields (with $p\leq n$). More accurately the $p$-uple $(Y_1,Y_2,\dots,Y_p)$ is an orthonormal basis of the subspace $\Delta$ it generates in $\fg$ and defines a left-invariant sub-Riemannian structure on $G$. This structure is Riemannian if $p=n$. It is sub-Riemannian if $p<n$ and we assume in that case that the distribution $\{Y_1,Y_2,\dots,Y_p\}$ is bracket generating. On the other hand let $\XC$ be a linear field on $G$. As previously we denote by $\DC=-\ad(\XC)$ the associated derivation of $\fg$ and by $\{\varphi_t\}_{t\in\R}$ the flow of $\XC$. On the context of sub-Riemannian geometry, Theorem \ref{iso} reads:

\begin{theorem} 
	\label{Main}
	If $\Delta$ is $\DC$-invariant and $\DC|_{\Delta}$ is skew-symmetric then $\{\varphi_t\}_{t\in\R}$ is a one-parameter group of isometries of $G$.
\end{theorem}

Conversely let $\{\varphi_t\}_{t\in\R}$ be a one-parameter group of isometries of a left-invariant structure on $G$. Let us assume that $\varphi_t$ is also an automorphism of $G$, for each $t\in\R$. In that case the one-parameter group $\{\varphi_t\}_{t\in\R}$ is generated by an infinitesimal automorphism, that is a linear field $\XC$, and to this linear field we can associate as usual the derivation $\DC=-\ad(\XC)$.

The first thing to notice is that each $\varphi_t$ transforms an admissible curve into an admissible curve hence preserves the distribution $\Delta$. Then $\varphi_t$ preserves the length of all admissible curves and its differential should preserve the length of the tangent vectors to admissible curves. Thanks to the fact that
$$\varphi_t\circ L_g=L_{\varphi_t(g)}\circ\varphi_t, \;\;\;\forall t\in\R, g\in G,$$
this implies that $(d\varphi_t)_e=\rme^{t\DC}$ is orthogonal on $\Delta$ for all $t$ and finally that $\DC$ is skew-symmetric w.r.t. the  inner product of $\Delta$.

It is not true that all isometries of left-invariant sub-Riemannian structures that fix the origin are automorphisms. For instance a counter-example was built by John Milnor on the rototranslation group (see \cite{Milnor76}). However it has been proven by Kivioja and Le Donne that in case when $G$ is nilpotent (and connected) the group of isometries is a Lie group of affine transformations (see \cite{KL16}). An affine transformation is by definition the composition of an automorphism with a left translation. Since all left translations are isometries, we can state:

\begin{theorem} 
	\label{Nilpotent}
	Let $G$ be a nilpotent connected Lie group and $(\Delta,\langle\ ,\ \rangle)$ be a sub-Riemannian structure on $G$, where $\Delta$ is a left-invariant bracket generating distribution and $\langle\ ,\ \rangle$ is a left invariant inner product on $\Delta$. The group of isometries of $G$ that preserves the origin is the Lie group of automorphisms of $G$ whose Lie algebra is the set of derivations of $\fg$ that preserve $\Delta$ and are skew-symmetric on $\Delta$.
\end{theorem}

\subsection{The semi-simple case}

It is assumed in this section that  $\fg$ is a semi-simple Lie algebra, and that $\fg=\mathfrak{l}\oplus\mathfrak{p}$ is a Cartan decomposition of $\fg$. Let $\theta$ be the related Cartan involution, that is the automorphism of $\fg$ whose restriction to $\mathfrak{l}$ is the identity $I_\mathfrak{l}$ and the restriction to $\mathfrak{p}$ is $-I_\mathfrak{p}$.

We assume in what follows that the inner product defining the sub-Riemannian metric is the restriction to the left-invariant distribution $\Delta$ of the associated inner product, that is:
$$
\langle Y,Z\rangle =-B(Y,\theta Z) \quad \mbox{ where $B$ stands for the Killing form}.
$$
All the derivations being inner, let $\DC=-\ad(X)$ for some $X\in \fg$. A straightforward computation using the properties of the Killing form shows that
$$
\forall Y,Z \in \fg \qquad \langle \DC Y,Z \rangle=-\langle Y, (\theta\circ \DC\circ\theta)\ Z\rangle.
$$
Let us assume that $\Delta$ is $\DC=-\ad(X)$-invariant. Then the restriction of $\DC$ is skew-symmetric on $\Delta$ if and only if it commutes with $\theta$. But another straightforward computation shows that:
$$
(\DC\circ\theta)Y=(\theta\circ \DC)Y \quad \mbox{and}\ \ (\DC\circ\theta)Z=(\theta\circ \DC)Z\Longrightarrow (\DC\circ\theta)[Y,Z]=(\theta\circ \DC)[Y,Z].
$$
Since $\Delta$ is assumed to be bracket generating we obtain that $\DC$, and not only its restriction to $\Delta$, must commute with $\theta$. In other words $\DC$ cannot be skew-symmetric on $\Delta$ without being skew-symmetric on $\fg$.

Let $Y\in \fg$. Then $(\DC\circ\theta)Y=[\theta Y,X]$ and $(\theta\circ \DC)Y=\theta[Y,X]=[\theta Y,\theta X]$. If these two expressions are equal for all $Y\in\fg$ then according to the properties of semi-simple Lie algebras we deduce $\theta X=X$ hence $X\in\mathfrak{l}$.

\begin{theorem} 
	\label{SemiSimple}
	Let $G$ be a semi-simple connected Lie group and $(\Delta,\langle\ ,\ \rangle)$ be a sub-Riemannian structure on $G$, where $\Delta$ is a left-invariant bracket generating distribution and $\langle\ ,\ \rangle$ is a left invariant inner product on $\Delta$. Let $\mathfrak{n}$ be the normalizer of $\Delta$ in $\fg$. The group of automorphic isometries is isomorphic to the Lie group whose Lie algebra is equal $\mathfrak{l}\cap\mathfrak{n}$.
\end{theorem}

\begin{proof}
Indeed the group of isometries that fix the origin is a Lie group (\cite{KL16}). The group of automorphic isometries is its intersection with the Lie group of automorphisms of $G$, hence a Lie group. Its Lie algebra is the set of derivations $\DC=-\ad(X)$ that preserve $\Delta$, that is $X\in \mathfrak{n}$, and that are skew-symmetric on $\Delta$ hence on $\fg$, that is $X\in\mathfrak{l}$.
\end{proof}

\end{document}